\documentclass[12pt]{article}
\usepackage{amssymb,amsmath,amsthm,epsfig}
\usepackage{latexsym, enumerate}
\usepackage{eepic}
\usepackage{epic}
\usepackage{graphicx}
\usepackage{color}
\usepackage{ifpdf}
\usepackage{subfigure}
\usepackage{tikz}
\usepackage{dsfont}
\usepackage{multirow}
\usepackage{makecell}
\usepackage{algorithm}
\usepackage{bm}
\usepackage{multirow}
\usepackage{epstopdf}

\usepackage{cite}
\usepackage{hyperref}
\usepackage{xcolor}
\hypersetup{
  colorlinks,
  linkcolor={blue!80!black},
  urlcolor={blue!80!black},
  citecolor={blue!80!black}
}
\usepackage[capitalize,noabbrev]{cleveref}

\theoremstyle{plain}
\newtheorem{lem}{Lemma}[section]
\newtheorem{thm}[lem]{Theorem}

\theoremstyle{definition}
\newtheorem{defn}{Definition}[section]

\theoremstyle{remark}
\newtheorem{rem}{Remark}[section]

\begin{document}
\title{\MakeUppercase {\large\bf On inverse problems for several coupled pde systems arising in mathematical biology} }

\author{
Ming-Hui Ding\thanks{Department of Mathematics, City University of Hong Kong, Kowloon, Hong Kong, China.\ \  Email: mingding@cityu.edu.hk}
\and
Hongyu Liu\thanks{Department of Mathematics, City University of Hong Kong, Kowloon, Hong Kong, China.\ \ Email: hongyu.liuip@gmail.com; hongyliu@cityu.edu.hk}
\and
Guang-Hui Zheng\thanks{School of Mathematics, Hunan University, Changsha 410082, China.\ \ Email: zhenggh2012@hnu.edu.cn; zhgh1980@163.com}
}

\date{}
\maketitle

\begin{center}{\bf ABSTRACT}
\end{center}\smallskip

In this paper, we propose and study several inverse problems of identifying/determining unknown coefficients for a class of coupled PDE systems by measuring the average flux data on part of the underlying boundary. In these coupled systems, we mainly consider the non-negative solutions of the coupled equations, which are consistent with realistic settings in biology and ecology. There are several salient features of our inverse problem study: the drastic reduction of the measurement/observation data due to averaging effects, the nonlinear coupling of multiple equations, and the non-negative constraints on the solutions, which pose significant challenges to the inverse problems. We develop a new and effective scheme to tackle the inverse problems and achieve unique identifiability results by properly controlling the injection of different source terms to obtain multiple sets of mean flux data. The approach relies on certain monotonicity properties which are related to the intrinsic structures of the coupled PDE system. We also connect our study to biological applications of practical interest.

\smallskip
{\bf keywords}: Nonlinear coupled parabolic systems; mathematical biology; inverse coefficient problems; unique identifiability; monotonicity; comparison principle.

%\begin{center}\footnotesize
%%\small
%\tableofcontents
%\end{center}

\section{Introduction}
\subsection{Mathematical setup and discussion of major findings}

Focusing mainly on the mathematics, but not the physical or biological applications, we first introduce the mathematical setup of the study. Let $\Omega$ be a bounded domain in $\mathbb{R}^N, N\geq1$ with a $C^{2+\gamma}$ boundary for some $\gamma\in(0,1)$. For each $T>0$, we set $D_T=\Omega\times (0,T],  S_T=\partial\Omega\times(0,T]$. We consider the following coupled nonlinear parabolic system with a Dirichlet boundary condition:
\begin{equation}\small
\label{utmt}
\begin{cases}
\displaystyle\frac{\partial u}{\partial t}-d_1\Delta u+\alpha\cdot\nabla u=l(x,t)u+a_1\bigg{(}-u^2-\sum_{k=1}^Mu^km^{M-k+1}\bigg{)}+q_1
\ \ \hspace*{1.0cm} &\mathrm{in}\  D_T,\medskip\\
\displaystyle \frac{\partial m}{\partial t}-d_2\Delta m+\beta\cdot\nabla m=v(x,t)m+a_2\bigg{(}-m^2-\sum_{k=1}^Mu^km^{M-k+1}\bigg{)}+q_2
\hspace*{0.65cm} &\mathrm{in}\ D_T,\medskip\\
\displaystyle  u=m=0\hspace*{3.7cm} &\mathrm{on}\  S_T,\medskip\\
u(x,0)=0,\  m(x,0)=0\ \hspace*{2.4cm} &\mathrm{in}\ \Omega,\medskip\\
\displaystyle u, m\geq 0\ \hspace*{2.4cm} &\mathrm{in}\ {\overline{D_T}},
\end{cases}
\end{equation}
where the diffusion coefficients $d_1, d_2$ are positive constants; the convection coefficients $\alpha, \beta$ are vectors of $\mathbb{R}^N$ with constant entries; $l(x,t)$, $v(x,t)\in C^{\gamma,\gamma/2}(\overline{D_T})$ are negative functions, and the functions $a_i$, $q_i\in C^{\gamma,\gamma/2}(\overline{D_T})$ for $i=1,2$. In Section~\ref{sect:nt1}, we shall define the function spaces used here. It is emphasised that $u$ and $m$ are required to be non-negative in ${\overline{D_T}}$. In related biological and ecology settings, $u$ and $m$ represent the probability density of some population, the non-negativity is crucial. More discussion on this aspect will be given in Section~\ref{sect:bm1}.

In this paper, we propose to study the inverse problem of identifying/determining the coefficients $l(x,t)$ and $v(x,t)$ in the coupled system \eqref{utmt} by knowledge of the average flux data of $u$ and $m$ on part of the boundary $\partial\Omega$. To that end, we introduce the following measurement maps for the system \eqref{utmt}:
\begin{align}\label{Fdata}
\mathcal{F}_{l,v}^+(a_1,a_2,q_1,q_2):=\int_{\Gamma_T}\frac{\partial u(x,t)}{\partial\nu}h(x,t)dxdt,\\
\label{Gdata}
\mathcal{G}_{l,v}^+(a_1,a_2,q_1,q_2):=\int_{\Gamma_T}\frac{\partial m(x,t)}{\partial\nu}g(x,t)dxdt,
\end{align}
where $a_1, a_2\in\mathcal{A}$ and $b_1, b_2\in\mathcal{B}$ with $\mathcal{A}$ and $\mathcal{B}$ signifying certain admissible classes of input parameters; $\Lambda$ is an open subset of $\partial\Omega$, ${\Gamma_T}=\Lambda\times (0,T]$; and $\nu(x)$ denotes the outward unit normal. In \eqref{Fdata}, the sign ``$+$" indicates that the measurement map is defined as associated with the non-negative solutions of the coupled system \eqref{utmt}. The measurement data \eqref{Fdata} and \eqref{Gdata} are the weighted integral data of the partial boundary, and the weight functions $g, h$ related to the measuring means \cite{ZD2020,DZ2021} shall be described in more details in what follows. The inverse problem can be expressed in the following general form
\begin{equation}\label{eq:ip1}
\mathcal{F}^+_{l,v}(a_1,a_2,q_1,q_2),\  \mathcal{G}^+_{l,v}(a_1,a_2,q_1,q_2),\ a_j\in\mathcal{A}, b_j\in\mathcal{B}, j=1,2 \longrightarrow l \ \mathrm{and} \ v.
\end{equation}
We are mainly concerned with the unique identifiability issue of the inverse problem \eqref{eq:ip1}. In a formal manner, our major finding can be roughly summarised into the following theorem.
\begin{thm}\label{thm1}
Suppose that $l_j, v_j\in \mathcal{S}$, $j=1, 2$, with $\mathcal{S}$ being a general a-priori function space.
Let $\mathcal{F}_{l_j,v_j}^+(a_1,a_2,q_1,q_2),\mathcal{G}_{l_j,v_j}^+(a_1,a_2,q_1,q_2)$ be the measurement maps associated to \eqref{utmt} for $j=1,2$.  If
\begin{align*}
\mathcal{F}^+_{l_1,v_1}(a_1,a_2,q_1,q_2)&=\mathcal{F}^+_{l_2,v_2}(a_1,a_2,q_1,q_2),\\  \mathcal{G}^+_{l_1,v_1}(a_1,a_2,q_1,q_2)&=\mathcal{G}^+_{l_2,v_2}(a_1,a_2,q_1,q_2),\
\end{align*}
for all $a_1, a_2\in\mathcal{A}$, $q_1, q_2\in\mathcal{B}$, then one has that
\begin{align*}
(l_1,v_1)=(l_2,v_2).
\end{align*}
\end{thm}

Indeed, in several different setups, we establish subtly different conditions for Theorem~\ref{thm1} to assure that the average flux data can uniquely identify/determine the coefficients $l, v$. Moreover, we can also consider the case that the homogeneous Dirichlet data on $S_T$ in \eqref{utmt} is replaced by the homogeneous Neumann data, and one measures the average Dirichlet data on $\Gamma_T$. In such a case, we can derive similar unique identifiability results for the associated inverse problems.

To provide a global view of our study, we next briefly discuss the main results of this paper, with proofs referring to Sections \ref{sect:3}--\ref{sect:5}. We first focus on the unique identification of the unknown coefficients $l(x,t)$ and $v(x,t)$ for three different coupled systems with a homogeneous Dirichlet boundary condition including the nonlinear coupled parabolic system, the linear coupled parabolic system, and the nonlinear coupled elliptic system. Secondly, we study the unique reconstruction of the coefficients for a coupled parabolic system with Neumann boundary conditions. The main results of the study are listed below:
\begin{itemize}
\item{For the inverse problem of the nonlinear coupled parabolic system \eqref{utmt}, the coefficients $l(x,t)$, $v(x,t)$  can be uniquely reconstructed by measuring $(\mathcal{F}^+_{l,v}$, $\mathcal{G}^+_{l,v})(a_1, a_2, q_1, q_2)$ with different inputs.}
\item{ For the inverse problem of a linear coupled parabolic system, we consider the inverse problem for the case with the coefficients $l(x),v(x)$ depending only on the spatial variable $x$.   The amount of average flux data required to reconstruct $(l(x), v(x))$ is greatly reduced compared to the amount of data required to determine the coefficients $l(x,t), v(x,t)$ (see Remark \ref{rem4.2}).}
\item{In the static case, the parabolic system is reduced to an elliptic system. For the inverse problem of the nonlinear coupled elliptic system, by measuring $(\mathcal{F}^+_{l,v},\mathcal{G}^+_{l,v})(a_1,a_2,q_1,q_2)$ with proper inputs, we can uniquely reconstruct the coefficients $l(x), v(x)$.}
\item{ We can extend the study by imposing homogeneous Neumann boundary condition on the nonlinear coupled parabolic systems \eqref{utmt}, and measuring the average Dirichlet data, and derive similar uniqueness results in identifying/determining the coefficients $l(x,t),v(x,t)$.}
\end{itemize}

\subsection{Background and motivation}\label{sect:bm1}

The research of coupled systems includes many models in classical scientific fields, such as biology\cite{May1976,S2018,Se2009}, ecology \cite{BB1984, LL1988,DY2002}, biochemistry \cite{M1977,F1977,MSR2010}, as well as physics and engineering \cite{GSS2006,GNH2009,ACR2007}. In biological evolution, the system \eqref{utmt} can portray the relationship between two competing species in an ecosystem, which is known as the Volterra-Lotka competition model. In this model, the right item $F_i\ (i=1,2)$ of system \eqref{utmt} is the more special form as follows \cite{M1975}:
\begin{align*}
F_1(x,t,u,m)&=l(x)u-a_1(x)u^2-b_1(x)um,\\
F_2(x,t,u,m)&=v(x)m-a_2(x)m^2-b_2(x)mu,
\end{align*}
where the positive function $a_i$ $(i=1,2)$ characterize the restriction effect between each population;
and the non-negative coefficient $b_i$ $(i=1,2)$ indicate the interaction between species; $l(x)$ and $v(x)$ denote the growth rates of the species. In particular, $u(x,t), m(x,t)$ represent the population densities, which are non-negative. In essence, the non-negative solution of the coupled system is closely related to the positivity or negativity of the coefficients of the right term of the equations; see e.g. \cite{P2012} and the references cited therein. In other classical models, $u$ and $m$ represent different quantities, depending on the physical setting and background from which the equations are derived, and they are both non-negative as well; see e.g. \cite{AAR2004,DLS2002,R2013}.

The problem we investigate is to uniquely identify $l(x,t)$ and $v(x,t)$ in system \eqref{utmt} by controlling the input coefficients $a_1, a_2$ and source $q_1, q_2$ to measure multiple sets the average flux data at the partial boundary. The study of this inverse problem is driven by practical applications. First, the coefficients $l(x,t), v(x,t)$ characterize certain properties of the system, say e.g., in the biological evolution of a population, $l(x,t)$ and $v(x,t)$ are the growth rates of the population. The reconstruction of the growth rates can predict the development of the population and adapt its behavior to a complex environment. Second, it is more natural to use boundary measurement data to study inverse problems, and the average flux data \eqref{Fdata},\eqref{Gdata} are readily measurable in practice \cite{DEM2009,MRMG2009}. Another practical scenario of our study is related to the control of biological population growth. In fact, we investigate the inverse problem by entering appropriate control coefficients $a_1, a_2$ and source terms $q_1,q_2$. This inverse problem is clearly related to the design of the growth of bacteria or yeast, by controlling the strength of the interaction of multiple populations (namely, the coefficient $a_1, a_2$ in system \eqref{utmt}) or by changing the PH of the environment ($q_1, q_2$ in system \eqref{utmt}) to quantify the growth patterns of populations to better understand the behavior of the ecosystem; see e.g. \cite{AL2023,HAM2022}.

Finally, we would like to briefly mention a few results on inverse problems for coupled parabolic systems in the literature.  Dou-Yamamoto \cite{DY2019} studied the inverse coefficient problem for two coupled Schr\"{o}dinger systems and obtained logarithmic stability results. Cristofol-Gaitan-Niinim$\mathrm{\ddot{a}}$ki-Poisson \cite{CGN2013} investigated an inverse problem of simultaneously identifying two discontinuous diffusion coefficients for a one-dimensional coupled parabolic system with one component observed. In \cite{RC2012}, a one-dimensional unique result is proved for the Lotka-Volterra competitive model corresponding to several non-constant coefficient inverse problems for two systems of parabolic equations. In these studies, the measurement/observation data used are all internal sub-domain data corresponding to a single event and the non-negativity of solutions does not play any role. In a recent paper \cite{LL2023}, the inverse boundary problem for a coupled parabolic system arising in biological applications was considered, and we shall discuss more on the technical differences in the next subsection.

\subsection{Technical developments and discussion}

The non-negativity constraint, and the drastic reduction of the measurement data due to the averaging effects as well as the partial-boundary measurement present significant challenges to the analysis of both the forward and inverse problems of the nonlinear coupled system.

In this paper, we first analyze the well-posedness of the forward problem for the nonlinear coupled parabolic system, the linear coupled parabolic system, and the nonlinear coupled elliptic system. In our analysis, we make essential use of the monotonic approach of the coupled system and its associated upper and lower solutions, which contributes to the study of the uniqueness of non-negative solutions to the forward problem. There are some related studies on the use of the method of upper and lower solutions as well as the associated monotone iteration to analyze the solutions of coupled systems, and we refer to \cite{AAR2004,ACR2007,BB1984}.

In the study of the inverse problem, we achieve the unique identifiability of the coefficients $l(x,t), v(x,t)$ or $l(x), v(x)$ by appropriately selecting the admissible classes $\mathcal{A}$ of the parameters $a_1, a_2$ and admissible classes $\mathcal{B}$ of source terms $q_1, q_2$. This method is different from constructing Complex-Geometric-Optics (CGO) solutions of coupled equations to fulfill the inverse problem \cite{LLX2009,LZ2022,LMZ2022}. In general, it is a high challenge to build non-negative CGO solutions. At present, we only know of non-negative solutions to parabolic equations established in \cite{LL2023}. Our proposed method precisely avoids the construction of non-negative CGO solutions of the coupled equations, and only needs to control the interaction coefficients $a_1,a_2$ and source terms $q_1, q_2$ to achieve the uniqueness of $l(x,t), v(x,t)$ while satisfying the non-negative constraints. The selection of the admissible class $\mathcal{A}$ for $a_1, a_2$ and the admissible class  $\mathcal{B}$ for $q_1, q_2$, is given in Section~\ref{sect:4}. In addition, this method also provides unique identifiability for the coupled system with a homogeneous Neumann boundary, and the results are similar to the proof for the homogeneous Dirichlet boundary condition; see Section~\ref{sect:5}.

The rest of the paper is organized as follows. In Section~\ref{sect:2}, we present some preliminaries and statements. Section~\ref{sect:3} is devoted to the study of the well-posedness of the forward problems. The proofs of the unique determination for the coupled systems are provided in Section~\ref{sect:4}. The unique determination results for the coupled parabolic system with the Neumann boundary are given in Section~\ref{sect:5}.

\section{Preliminaries and statements}\label{sect:2}
\subsection{Notations and basic setting}\label{sect:nt1}
For a fixed $0<\gamma<1$, the H\"{o}lder space $C^{k+\gamma}(\overline{\Omega})$  consists of those functions $u$ that are $k$-times continuously differentiable and whose $k^{th}$-partial derivatives are H\"{o}lder continuous with the H\"older-index being $\gamma$. We denote $D^l u:= \frac{\partial^{l_1}}{\partial x_1^{l_1}}...\frac{\partial^{l_n}}{\partial x_n^{l_n}}u$ for $l=(l_1,...,l_n)\in \mathbb{N}^n$ being a multi-index of order $|l|\leq k$. The associated norm is defined as
\begin{align*}
\parallel u\parallel_{C^{k+\gamma}(\overline{\Omega})}:=\sum_{|l|\leq k}\parallel D^l u\parallel_{\infty}+\sum_{\mid l\mid=k}\sup_{x\neq y}\frac{\mid D^l u(x)-D^lu(y)\mid}{\mid x-y\mid^\gamma}.
\end{align*}
Given $T\in(0,\infty]$, we consider the space $C^{k+\gamma,\frac{k+\gamma}{2}}(\overline{D_T})$ if $D^lD^j_t u$ exists and is H$\mathrm{\ddot{o}}$lder continuous with the exponent $\gamma$ in $x$ and $\frac{k+\gamma}{2}$ in $t$ for all $l\in \mathbb{N}^n, j\in \mathbb{N}$ with $|l|+2j\leq k$. The norm is defined as
\begin{align*}
\parallel u\parallel_{C^{2+\gamma,1+\gamma/2}(\overline{D_T})}:&=\sum_{|\gamma|+2j\leq k}\parallel D^lD^j_t u\parallel_{\infty}+\sum_{\mid\gamma\mid+2j=k}\sup_{t,x\neq y}\frac{\mid  u(x,t)-u(y,t)\mid}{\mid x-y\mid^\gamma}\\
&+\sum_{\mid\gamma\mid+2j=k}\sup_{x,t\neq t'}\frac{\mid  u(x,t)-u(x,t')\mid}{\mid t-t'\mid^{\gamma/2}}.
\end{align*}

\subsection{Well-posedness conditions}
In this subsection, we present the definition of the quasimonotone function, and upper, lower solutions of the system, which are crucial in proving the existence and uniqueness of positive solutions to the forward problem.

First, in order to appeal for a more general study, the right hand side (RHS) term $\mathbf{F}=(F_1,F_2)$ in the coupled system \eqref{utmt} can be replaced by the following general form:
\begin{align}\label{F1}
F_1(x,t,u,m)&=l(x,t)u+f_1(x,t,u,m),\\
\label{F2}
 F_2(x,t,u,m)&=v(x,t)m+f_2(x,t,u,m),
\end{align}
 where $(f_1, f_2)$ are continuously differentiable in $(u, m)$ for all $(u, m)\in J_1\times J_2$, with $J_1\times J_2$ being a bounded subset in $\mathbb{R}^2$. Then we introduce the definition of quasi-monotone function $(F_1, F_2)$ as follows.

\begin{defn}
Let the function $\mathbf{F}=(F_1, F_2)$ be given in \eqref{F1} and \eqref{F2}. If the following conditions are fulfilled,
\begin{align*}
\partial F_1/\partial m\geq 0,\ \ \partial F_2/\partial u\geq 0 \ \ \mathrm{for}\ (u, m)\in J_1\times J_2,
\end{align*}
then $\mathbf{F}$ is said to be quasi-monotonically nondecreasing. If
\begin{align*}
\partial F_1/\partial m\leq 0,\ \ \partial F_2/\partial u\leq0 \ \ \mathrm{for}\ (u, m)\in J_1\times J_2,
\end{align*}
then $\mathbf{F}$ is said to be quasi-monotonically nonincreasing.
\end{defn}

Next, we introduce the upper and lower solutions of the system \eqref{utmt}, denoted by $\mathbf{U}=(\tilde{u}, \tilde{m})$ and $\mathbf{V}=(\hat{u}, \hat{m})$, respectively, and state that its boundary value and initial value satisfy the following inequalities
\begin{eqnarray}\label{bini}
\begin{aligned}
\tilde{u}\geq0\geq\hat{u},\ \ &\tilde{m}\geq0\geq\hat{m} \ \hspace*{4.25cm} \ \mathrm{on}\ S_T,\\
\tilde{u}(x,0)\geq u(x,0)\geq\hat{u}(x,0),\ \ &\tilde{m}(x,0)\geq m(x,0)\geq\hat{m}(x,0)\ \hspace*{1.3cm}\ \mathrm{in}\ \Omega.
\end{aligned}
\end{eqnarray}
 We provide the precise definition of upper and lower solutions (cf. \cite{P2012}).
\begin{defn}\label{uplow}
Let $\mathbf{U}=(\tilde{u}, \tilde{m})\in C^{2+\gamma,1+\gamma/2}(\overline{{D_T}})\times C^{2+\gamma,1+\gamma/2}(\overline{{D_T}})$, and $\mathbf{V}=(\hat{u}, \hat{m})\in C^{2+\gamma,1+\gamma/2}(\overline{{D_T}})\times C^{2+\gamma,1+\gamma/2}(\overline{{D_T}})$ be solutions to \eqref{utmt} and satisfy the relation $\mathbf{U}\geq \mathbf{V}$. Here and also in what follows,  $\mathbf{U}\geq \mathbf{V}$ is an ordering relation in the component-wise sense, i.e., $\tilde{u}\geq\hat{u}$ and $\tilde{m}\geq\hat{m}$. It is also assumed that the inequalities in \eqref{bini} hold and moreover, if $(F_1, F_2)$ is quasi-monotonically nondecreasing,
 \begin{align*}
\tilde{u}_t-d_1\Delta\tilde{u}+\alpha\cdot\nabla \tilde{u}-F_1(x,t,\tilde{u},\tilde{m})&\geq 0\geq \hat{u}_t-d_1\Delta\hat{u}+\alpha\cdot\nabla \hat{u}-F_1(x,t,\hat{u},\hat{m}),\\
\tilde{m}_t-d_2\Delta\tilde{m}+\beta\cdot\nabla \tilde{m}-F_2(x,t,\tilde{u},\tilde{m})&\geq 0\geq \hat{m}_t-d_2\Delta\hat{m}+\alpha\cdot\nabla \hat{m}-F_2(x,t,\hat{u},\hat{m}),
\end{align*}
and if $(F_1, F_2)$ is quasi-monotonically nonincreasing,
\begin{align*}
\tilde{u}_t-d_1\Delta\tilde{u}+\alpha\cdot\nabla \tilde{u}-F_1(x,t,\tilde{u},\hat{m})&\geq 0\geq \hat{u}_t-d_1\Delta\hat{u}+\alpha\cdot\nabla \hat{u}-F_1(x,t,\hat{u},\tilde{m}),\\
\tilde{m}_t-d_2\Delta\tilde{m}+\beta\cdot\nabla \tilde{m}-F_2(x,t,\hat{u},\tilde{m})&\geq 0\geq \hat{m}_t-d_2\Delta\hat{m}+\alpha\cdot\nabla \hat{m}-F_2(x,t,\tilde{u},\hat{m}).
\end{align*}
Then $\mathbf{U}$ and $\mathbf{V}$ are respectively called the ordered upper and lower solutions of \eqref{utmt}.
\end{defn}

\section{Well-posedness of the forward problem}\label{sect:3}
In this section, we study the well-posedness of the forward problem for the nonlinear  coupled parabolic system, linear coupled parabolic system, and nonlinear coupled elliptic  system, i.e., the existence and uniqueness of a positive solution (coexistence state) to the respective forward problem. We first define the following sector
\begin{align*}
\langle\mathbf{U}, \mathbf{V}\rangle\equiv\{(u,m)\in  C^{2+\gamma,1+\gamma/2}(\overline{D}_T)\times C^{2+\gamma,1+\gamma/2}(\overline{D}_T); (\hat{u}, \hat{m})\leq(u,m)\leq(\tilde{u}, \tilde{m})\},
\end{align*}
where $(\hat{u}, \hat{m}), (\tilde{u}, \tilde{m})$ are the lower  and upper solutions of the Definition \ref{uplow}, respectively. Moreover, for any $G\in C^{\gamma,\gamma/2}(\overline{D_T})$ we denote
\begin{align*}
G_L:=\mathrm{ess}\inf_{D_T} G,\ \ G_S:=\mathrm{ess}\sup_{D_T}G.
\end{align*}
\subsection{Nonlinear coupled parabolic equations}
For the forward problem, we can consider a more general form of the coupled system \eqref{utmt} with the RHS term $F_i, (i=1,2)$  given by the following form
\begin{equation}\label{f12}
\begin{aligned}
F_1(x,t,u,m)=l(x,t)u-a_1(x,t)u^2-\sum_{k=1}^Mb_1^{(k)}(x,t)u^km^{M-k+1}+q_1(x,t),\\
F_2(x,t,u,m)=v(x,t)m-a_2(x,t)m^2-\sum_{k=1}^Mb_2^{(k)}(x,t)m^ku^{M-k+1}+q_2(x,t),
\end{aligned}
\end{equation}
where $M\geq 1$, and the coefficients $a_i, b_i^{(k)},q_i, i=1,2, k=1,...,M$, belong to $C^{\gamma,\gamma/2}(\overline{{D_T}})$. Next, we first show an auxiliary lemma of the coupled system, and its proof can be conveniently found in \cite{P2012}.

\begin{lem}\label{lem1}
Let $(\tilde{u}, \tilde{m})$ and $(\hat{u}, \hat{m})$ be respectively the ordered upper and lower solutions of \eqref{utmt}, and $\mathbf{F}=(F_1, F_2)$ be a $C^1$-function and quasi-monotonically nonincreasing (or nondecreasing) in $\langle\mathbf{U}, \mathbf{V}\rangle$. Then the equation \eqref{utmt} has a unique solution $(u, m)$ in  $\langle\mathbf{U}, \mathbf{V}\rangle$.
\end{lem}

Using Lemma \ref{lem1},  we can show the uniqueness result for a positive solution of the system \eqref{utmt} when $(F_1, F_2)$ is given in the form \eqref{f12}.
\begin{thm}\label{forward}
Let $(F_1, F_2)$ be given by \eqref{f12} with  $a_i, b_i^{(k)},q_i \in C^{\gamma,\gamma/2}(\overline{{D_T}})$ for $i=1,2, k=1,..., M$. If $a_i, q_i$ are positive functions, and $b_i^{(k)}$ is non-negative, then the system \eqref{utmt} has a unique positive solution $(u, m)$ and satisfies
\begin{align*}
0<u\leq\tilde{u},\ \ 0<m\leq\tilde{m},
\end{align*}
where $(\tilde{u}, \tilde{m})$ is an upper solution.
\end{thm}
\begin{proof}
Since $a_i, q_i$ are positive functions, $b_i^{(k)}$ is nonnegative, the function $(F_1, F_2)$ is quasi-monotonically nonincreasing in $\mathbb{R}^+\times \mathbb{R}^+$ and the trivial function $(0,0)=(\hat{u}, \hat{m})$ is a lower solution. Set the constant function $(\tilde{u}, \tilde{m})$ as a positive upper solution. If
\begin{align*}
\tilde{u}_t-d_1\Delta \tilde{u}+\alpha\cdot\nabla \tilde{u}\geq F_1(x,t,\tilde{u},\hat{m}),\\
\tilde{m}_t-d_2\Delta \tilde{m}+\beta\cdot\nabla \tilde{m}\geq F_2(x,t,\hat{u},\tilde{m}),
\end{align*}
then we have
\begin{align*}
-l(x,t)\tilde{u}+a_1(x,t)\tilde{u}^2-q_1(x,t)\geq0,\\
-v(x,t)\tilde{m}+a_2(x,t)\tilde{m}^2-q_2(x,t)\geq0.
\end{align*}
According to $l(x,t), v(x,t) <0$, we find
\begin{align*}
-l_{S}\tilde{u}+(a_1)_L\tilde{u}^2-(q_{1})_S\geq 0,\\
-v_{S}\tilde{m}+(a_2)_L\tilde{m}^2-(q_{2})_S\geq 0.
\end{align*}
Setting
\begin{align*}
 \tilde{u}>\frac{l_{S}+\sqrt{l_{S}^2+4(a_1)_L (q_1)_S}}{2(a_1)_L}>0,\\
 \tilde{m}>\frac{v_{S}+\sqrt{v_{S}^2+4(a_2)_L (q_2)_S}}{2(a_2)_L}>0,
\end{align*}
then $(\tilde{u}, \tilde{m})$ is a positive upper solution. The conclusion of the theorem follows from Lemma \ref{lem1}. Due to $q_i>0, i=1,2$, we obtain $(0,0)<(u,m)\leq(\tilde{u},\tilde{m})$. The proof is complete.
\end{proof}
\subsection{Linear coupled parabolic equations}
In this subsection,  we consider the well-posedness of the forward problem for a linearly coupled parabolic system. Here, we would like to point out that in addition to the biological application, it is also a useful model in gas-liquid reaction theory as well as in molecular multiphoton transitions \cite{GSS2006}. For linearly coupled system,  the RHS term $F_i\ (i=1,2)$ has the form
\begin{equation}\label{linf}
\begin{split}
&F_1=l(x,t)u+b_1(x,t)m+q_1(x,t),\\
&F_2=v(x,t)m+b_2(x,t)u+q_2(x,t),
\end{split}
\end{equation}
where $b_i, q_i\in C^{\gamma,\gamma/2}(\overline{D_T})$ for $i=1, 2$.

\begin{thm}\label{thm-lin}
Let $(F_1, F_2)$ be given by \eqref{linf} with $b_i, q_i\in C^{\gamma,\gamma/2}(\overline{{D_T}})$ for $i=1,2$. If  $b_i$ is a non-negative function, and $q_i$ is positive, then the system \eqref{utmt} has a unique positive solution $(u, m)$ and satisfies
\begin{align*}
0<u\leq\tilde{u},\ \ 0<m\leq\tilde{m},
\end{align*}
where $(\tilde{u}, \tilde{m})$ is an upper solution.
\end{thm}
\begin{proof}
 Since $b_i\geq 0$, and $q_i$ is a positive function, the function $(F_1, F_2)$ is quasi-monotonically nondecreasing in $\mathbb{R}^+\times \mathbb{R}^+$ and $(0,0)=(\hat{u}, \hat{m})$ is a lower solution. Consider the solution $p_i$ of the following linear equation
\begin{align*}
&\displaystyle (p_1)_t-d_1\Delta p_1+\alpha\cdot\nabla p_1=q_1\ \ \mathrm{in}\ D_T,\\
&\displaystyle (p_2)_t-d_2\Delta p_2+\beta\cdot\nabla p_2=q_2\ \ \mathrm{in}\ D_T,
\end{align*}
where the boundary and initial conditions are the same as those in \eqref{utmt}. By the maximum principle, we have $p_i\geq 0$ in $\overline{D_T}$. For some positive constants $\rho$ and $\delta$, we define
\begin{align*}
 \tilde{u}=p_1+\rho e^{\delta t},\\
  \tilde{m}=p_2+\rho e^{\delta t}.
 \end{align*}
 Notice that $p_i$ is bounded in $\overline{D_T}$, there exists a adequate large $\delta$ for any $\rho>0$, such that
\begin{align*}
\delta\rho\geq (lp_1+b_1p_2)e^{-\delta t}+(l+b_1)\rho,\\
\delta\rho\geq (vp_2+b_2p_1)e^{-\delta t}+(v+b_2)\rho,
\end{align*}
then $(\tilde{u}, \tilde{m})$ is a positive upper solution. From Lemma \ref{lem1} it follows that $(0,0)\leq(u,m)\leq(\tilde{u},\tilde{m})$. Due to $q_i>0$ for $i=1, 2$, we have $(0,0)<(u,m)\leq(\tilde{u},\tilde{m})$. The proof is complete.
\end{proof}

\begin{rem}
For the coupled system \eqref{utmt} with the function $(F_1, F_2)$ given by \eqref{f12} or \eqref{linf}, if the Dirichlet boundary condition is replaced to be a Neumann boundary condition or a Robin boundary condition, the existence and uniqueness of a positive solution to the coupled system can be proved in a similar manner.
\end{rem}
\subsection{Coupled elliptic equations}
In this subsection, we consider a coupled elliptic system. For this system, the analysis of the existence and uniqueness of a positive solution to the coupled elliptic equations is further complicated by the absence of the term for the time derivative.
\begin{equation}
\label{umell}
\begin{cases}
\displaystyle -d_1\Delta u-l(x)u=-b_1\sum_{k=1}^Mu^km^{M-k+1}+q_1(x)
\hspace*{1.0cm}\ &\mathrm{in}\ \Omega,\medskip\\
\displaystyle -d_2\Delta m-v(x)m=-b_2\sum_{k=1}^Mu^km^{M-k+1}+q_2(x)
\hspace*{.7cm} &\mathrm{in}\ \Omega,\medskip\\
\displaystyle  u=m=0\hspace*{7.1cm} &\mathrm{on}\ \partial\Omega,\medskip\\
\displaystyle u>0,\  m>0\ \hspace*{8.4cm} &\mathrm{in}\ \Omega,
\end{cases}
\end{equation}
where $b_i\in\mathbb{R}$, $q_i\in C^{\gamma}(\overline{\Omega})$ for $i=1, 2$, $l(x), v(x)\in C^{\gamma}(\overline{\Omega})$. Let the RHS term $F_i$ of the system \eqref{umell} be given by the following form
\begin{align}\label{Fi}
F_i(x,u,m)=b_if(x,u,m)+q_i(x),\ i=1,2,
\end{align}
where $f$ is given by
\begin{align*}
f(x,u,m)=-\sum_{k=1}^Mu^km^{M-k+1}.
\end{align*}

Next, we provide the following uniqueness theorem for the positive solution of the steady-state problem \eqref{umell}.
\begin{thm}
Assume that $(F_1, F_2)$ is given by \eqref{Fi} with $b_i\in \mathbb{R}, q_i\in C^{\gamma, \gamma/2}$ for $i=1, 2$. If $b_i$ is a positive constant, and $q_i$ is a positive function, then there exists a unique positive solution $(u_s,v_s)$ to \eqref{umell} such that
\begin{align*}
(0,0)< (u_s, m_s)\leq (w_1^*, w_2^*),
\end{align*}
where $(w^*_1,w^*_2)$ is the positive solution of the following equations %\eqref{wi}
\begin{align*}
\begin{cases}
-d_1\Delta w_1-l(x)w_1(x)=q_1(x) \ \  &\mathrm{in}\ \  \Omega,\medskip\\
-d_2\Delta w_2-v(x)w_2(x)=q_2(x) \ \ &\mathrm{in}\ \  \Omega,\medskip\\
w_1=w_2=0\ \ \ &\mathrm{on}\ \partial\Omega.
\end{cases}
\end{align*}
\end{thm}
\begin{proof}
As $b_i$ is a positive constant, $q_i$ is a positive function,  the function $(F_1,F_2)$ is quasi-monotonically nonincreasing in $[0,\rho_1]\times[0,\rho_2]$ for $w_i^*\leq \rho_i$, and the following inequality holds
\begin{align*}
-d_1\Delta w_1^*-l(x)w_1^*- F_1(x,w_1^*,0)\geq 0,\\
-d_2\Delta w_2^*-v(x)w_2^*- F_2(x,0,w_2^*)\geq 0.
\end{align*}
  Thus, we obtain that the pairs $(w_1^*, w_2^*)$ and $(0,0)$ are ordered upper and lower solutions of the steady-state problem \eqref{umell}, respectively. The time-dependent solutions $(\overline{U}, \underline{V})$ and $(\underline{U}, \overline{V})$ of the following equation
\begin{align}
\label{um}
\begin{cases}
\displaystyle u_t-d_1\Delta u-l(x)u=-b_1\sum_{k=1}^Mu^km^{M-k+1}+q_1(x)
\hspace*{1.0cm}\ &\mathrm{in}\  D_T,\medskip\\
\displaystyle m_t-d_2\Delta m-v(x)m=-b_2\sum_{k=1}^Mu^km^{M-k+1}+q_2(x)
\hspace*{.7cm} &\mathrm{in}\ D_T,\medskip\\
\displaystyle u=m=0\hspace*{7.1cm} &\mathrm{on}\ S_T,\medskip\\
\displaystyle u(x,0)=u_0,\  m(x,0)=m_0\ \hspace*{4.8cm} &\mathrm{in}\ \Omega,
\end{cases}
\end{align}
correspond to $(u_0, m_0)=(w_1^*,0)$ and $(u_0,m_0)=(0, w_2^*)$, respectively. As $t\rightarrow \infty$, the solutions $(\overline{U}, \underline{V})$ and $(\underline{U}, \overline{V})$ monotonically converge to the steady state solutions $(\overline{U}_s, \underline{V}_s)$ and $(\underline{U}_s, \overline{V}_s)$ of \eqref{umell}, respectively, and satisfy (see Theorem 10.4.3 of \cite{P2012})
\begin{align*}
0\leq \underline{U}_s\leq \overline{U}_s\leq w_1^*,\\
0\leq \underline{V}_s\leq \overline{V}_s\leq w_2^*.
\end{align*}
We now claim $\underline{U}_s=\overline{U}_s$ and $\underline{V}_s=\overline{V}_s$. To verify this assertion, letting $\xi_1=\overline{U}_s-\underline{U}_s$, $\xi_2=\overline{V}_s-\underline{V}_s$, we have
\begin{equation}\label{xi1}
\begin{split}
-d_1\Delta\xi_1-l(x)\xi_1=b_1(f(x,\overline{U}_s, \underline{V}_s)-f(x,\underline{U}_s, \overline{V}_s)),\\
-d_1\Delta\xi_2-v(x)\xi_2=b_2(f(x,\underline{U}_s, \overline{V}_s)-f(x,\overline{U}_s, \underline{V}_s)).
\end{split}
\end{equation}
The first equation of \eqref{xi1} is multiplied by $b_2$ and the second by $b_1$, followed by addition, we find
\[
-\Delta W=b_2l(x)\xi_1+ b_1v(x)\xi_2\leq 0,
 \]
where $W=b_2 d_1\xi_1+b_1 d_2\xi_2$, and $W=0$ on $\partial\Omega$. According to the maximum principle, we have $W\leq 0$ in $\overline{\Omega}$. The nonnegative property of $\xi_1$ and $\xi_2$ implies that $W=0$. This shows that $\overline{U}_s=\underline{U}_s$ and $\overline{V}_s=\underline{V}_s$. By \cite[Corollary 10.4.2]{P2012}, if and only if $\underline{U}_s=\overline{U}_s$ and $\underline{V}_s=\overline{V}_s$, the system \eqref{umell} admits a unique steady-state solution, and satisfies $(0,0)< (u_s, m_s)\leq (w_1^*, w_2^*)$. The proof is complete.
\end{proof}

\section{Unique identifiability results of the inverse problems}\label{sect:4}
\subsection{Admissible class}
In this subsection,  we introduce the admissible conditions on the coefficients $a_1, a_2$ and source terms $q_1, q_2$. The choice of admissible class is crucial for the inverse problem, and we first provide the definitions of the admissible class as follows.
\begin{defn}
The function set $\{a^n(x,t)\}_{n=1}^\infty$ is admissible, denoted by
$\{a^n\}_{n=1}^{\infty}\in \mathcal{A}$, if it satisfies the following conditions:

(i) $a^n\in C^{\gamma,\gamma/2}(\overline{D_T})$ and $a^n>0$ for any $n\in\mathds{N}$.

(ii)$\lim\limits_{n\rightarrow\infty}a^n=0$.
\end{defn}
\begin{defn}
The function set $\{q^n(x,t)\}_{n=1}^\infty$ is admissible, denoted by
$\{q^n\}_{n=1}^{\infty}\in \mathcal{B}$, if the following  conditions are fulfilled:

 (i) $q^n\in C^{\gamma,\gamma/2}(\overline{D_T})$ and $q^n>0$ for any $n\in\mathds{N}$.

 (ii) $\{q^n(x,t)\}_{n=1}$ is a complete set in $L^2(\overline{D_T})$.
\end{defn}

To make these definitions clearer, we have the following remarks.
\begin{rem}
For the admissible classes $\mathcal{A}$ and $\mathcal{B}$, we restrict those functions to be positive. This restriction ensures the positivity of the solution for a coupled system. In the analysis of the well-posedness of the forward problems in Section \ref{sect:3}, the positivity of the coefficients $a_1,a_2$ and the source terms $q_1, q_2$ determine the positivity of the solution in a coupled system.
\end{rem}

\begin{rem}
For the admissible class $\mathcal{B}$, we require that the function set $\{q^n\}_{n=1}^\infty$ is positive and complete in $L^2(\overline{D_T})$. We would like to emphasize that such sequences exist. In fact, without loss of generality and by rigid transformations if necessary, we can assume that the domain $\Omega$ is contained in the first quadrant, i.e. $\Omega\Subset\{x=(x^1, x^2,\ldots, x^N)\in\mathbb{R}^N; x^j\geq 0, j=1,2,\ldots, N\}$. In such a case, the polynomial functions $\{1, x, t, tx,...\}$ form a complete set in $L^2(\overline{D_T})$ and the positivity of those base functions are obviously guaranteed.
\end{rem}

\subsection{Uniqueness results for inverse problems associated with the nonlinear coupled parabolic system}
In this subsection, we present the unique identifiability result for the inverse problem associated with the coupled system \eqref{utmt}. In this inverse problem, it is required to obtain multiple sets of average flux data $(\mathcal{F}_{l,v}^+, \mathcal{G}_{l,v}^+)$ to achieve the uniqueness.

\begin{thm}\label{thm1-1}
Let $\mathcal{F}^+_{l_i,v_i}$ and $\mathcal{G}^+_{l_i,v_i}$ be the measurement maps associated to \eqref{utmt} for $i=1,2$, and the weight functions $h,g\in C_0^{\gamma,\gamma/2}(\Gamma_T)$ be nonnegative and nonzero. Assume that $v,l\in C^{\gamma,\gamma/2}(\overline{D_T})$ are negative functions. If
\begin{equation}
\label{umtfg3}
\begin{aligned}
\mathcal{F}_{l_1,v_1}^+(a_1^s,a_2^s,q_1^n,q_2^n)&=\mathcal{F}_{l_2,v_2}^+(a_1^s,a_2^s,q_1^n,q_2^n),\ \\
\mathcal{G}_{l_1,v_1}^+(a_1^s,a_2^s,q_1^n,q_2^n)&=\mathcal{G}_{l_2,v_2}^+(a_1^s,a_2^s,q_1^n,q_2^n),\
\end{aligned}
\end{equation}
for $a_i^s\in\mathcal{A}$, $q_i^n\in \mathcal{B}$ with $s, n\in\mathds{N}$, then it holds that
$$l_1(x,t)=l_2(x,t), v_1(x,t)=v_2(x,t)\
 \ in\ D_T.$$
\end{thm}

\begin{proof}
 Since the process of obtaining the uniqueness of $v(x,t)$ is similar to the one of $l(x,t)$, we only show the proof of $l(x,t)$. Set $(u_{s,n}(x,t;l,v), m_{s,n}(x,t;l,v))$ to be the solution of \eqref{utmt} corresponding to the input functions $a_1^s,a_2^s,q_1^n,q_2^n$.
Notice that $h$ is a nonnegative and nonzero function, we set $h_0=h$ on $\Gamma_T$ and $h_0=0$ on $S_T\backslash\Gamma_T$. Let $ w(x,t;l_1)$ be the solution of the following adjoint problem
\begin{eqnarray}
\label{w1}
\begin{cases}
\displaystyle-\frac{\partial w}{\partial t}-d_1\Delta w-\alpha\cdot \nabla w-l_1(x,t)w=0\ \ \hspace*{1.5cm} &\text{in}\ \ D_T,\medskip\\
\displaystyle w(x,t)=h_0\ \ \hspace*{6.0cm}  &\text{on}\ \ S_T,\medskip\\
\displaystyle w(x,T)=0\ \ \hspace*{6.0cm} &\text{in}\ \ \Omega.\\
\end{cases}
\end{eqnarray}
From the first equation of \eqref{utmt}, \eqref{w1} and the Green's formula, we compute
\allowdisplaybreaks
\begin{align*}
&\int_{0}^T \int_{\Omega}\bigg{(}-a_1^su_{s,n}^2-a_1^s\sum_{k=1}^Mu_{s,n}^km_{s,n}^{M-k+1}+q_1^n\bigg{)}w(x,t;l_1)dxdt\\
=&\int_{0}^T\int_{\Omega}\bigg{(}\frac{\partial u_{s,n}}{\partial t}- d_1\Delta
u_{s,n}-l_1(x,t)u_{s,n}+\alpha\cdot\nabla u_{s,n}\bigg{)}w(x,t;l_1)dxdt\\
=&\int_{0}^T\int_{\Omega}\bigg{(}-\frac{\partial w}{\partial t}-d_1\Delta w-l_1(x,t)w-\alpha\cdot\nabla w\bigg{)}u_{s,n}dxdt\\
&-\int_0^T\int_{\partial\Omega} \frac{\partial u_{s,n}(x,t;l_1,v_1)}{\partial \nu}w(x,t;l_1)dxdt\\
=&-\int_0^T\int_{\Lambda}\frac{\partial u_{s,n}(x,t;l_1,v_1)}{\partial \nu(x)}h(x,t)dxdt.
\end{align*}
Similarly, for $w(x,t;l_2)$, we find
\begin{eqnarray*}
\int_{0}^T \int_{\Omega}\bigg{(}-a_1^su_{s,n}^2-a_1^s\sum_{k=1}^Mu_{s,n}^km_{s,n}^{M-k+1}+q_1^n\bigg{)}w(x,t;l_2)dxdt=
-\int_0^T\int_{\Lambda}\frac{\partial u_{s,n}(x,t;l_2,v_2)}{\partial \nu(x)}hdxdt.
\end{eqnarray*}
From the condition given in \eqref{umtfg3}, we obtain  that
\begin{eqnarray}
\label{njum}
\begin{aligned}
&\int_{0}^T \int_{\Omega}\bigg{(}-a_1^{s}u_{s,n}^2-a_1^{s}\sum_{k=1}^Mu_{s,n}^km_{s,n}^{M-k+1}+q_1^n\bigg{)}w(x,t;l_1)dxdt\\
=&\int_{0}^T \int_{\Omega}\bigg{(}-a_1^{s}u_{s,n}^2-a_1^{s}\sum_{k=1}^Mu_{s,n}^km_{s,n}^{M-k+1}+q_1^n\bigg{)}w(x,t;l_2)dxdt.
\end{aligned}
\end{eqnarray}
Notice that $\{a_1^s\}_{s=1}^{\infty}\in\mathcal{A}$ and $\{q_1^n\}_{n=1}^{\infty}\in\mathcal{B}$ are uniformly bounded sequences, combing this fact with Theorem \ref{forward}, there exist a constant upper solution $(\tilde{u}_{s,n},\tilde{m}_{s,n})$ of the system \eqref{utmt} for any $s,n\in\mathds{N}$,  thus we have
\begin{align*}
\bigg{|}\bigg{(}-a_1^{s}u_{s,n}^2-a_1^{s}\sum_{k=1}^Mu_{s,n}^km_{s,n}^{M-k+1}+q_1^n\bigg{)}w(x,t;l_i)\bigg{|}
< C,\ \ i=1,2,
\end{align*}
where $C$ is a constant. Let $s\rightarrow\infty$ in \eqref{njum}, and then by the dominated convergence theorem, we derive
\begin{eqnarray*}
\label{jum}
\begin{aligned}
\int_{0}^T \int_{\Omega}q_1^n(x,t)w(x,t;l_1)dxdt
=\int_{0}^T \int_{\Omega}q_1^n(x,t)w(x,t;l_2)dxdt.
\end{aligned}
\end{eqnarray*}
 From the completeness of $\{q_1^n\}_{n=1}^{\infty}$, we deduce
\begin{align}\label{w1=w2}
w(x,t;l_1)=w(x,t;l_2)\ \mathrm{for} \ (x,t)\in D_T.
\end{align}
The equations of \eqref{w1} corresponding to $l_1(x,t), l_2(x,t)$, respectively, followed by subtraction, and using \eqref{w1=w2}, we find
\begin{eqnarray*}
\label{lw}
(l_1-l_2) w(x,t;l_1)=0\ \mathrm{for}\ (x,t)\in D_T.
\end{eqnarray*}
 By the maximum principle, we have $w(x,t;l_1)> 0$ in $D_T$. Thus $l_1=l_2$ in $D_T$. The proof is complete.
\end{proof}

\begin{rem}
The uniqueness result for the inverse problem in Theorem \ref{thm1-1} can also be extended to the coupled parabolic system with a general elliptic operator as follows:
\begin{align*}
\mathcal{L}_1u=-\sum_{i,j=1}^{N}\left(c^{ij}u_{x_{i}}\right)_{x_{j}}+\sum_{i=1}^{N}d^{i}u_{x_i},
\end{align*}
where the coefficients $c^{ij}$ and $d^{i}$ are given. The unique result for determining the unknown coefficients can be obtained under the same conditions of Theorem \ref{thm1-1}. The crucial issue in the proof is the existence of a maximum principle for the conjugate equation. We refer to \cite{E2009} for the maximum principle of a parabolic equation with the general elliptic operator $\mathcal{L}_1$.
\end{rem}

\subsection{Uniqueness results for inverse problems associated with the linear coupled parabolic system}
In this subsection, we study the inverse problem of the linear coupled system  \eqref{utmt} with the RHS term $(F_1, F_2)$ given by \eqref{linf}. In this system, we consider  identifying $l(x)$ and $v(x)$ that depend only on the spatial variable $x$. Compared to determining the coefficients  $l(x,t)$ and $v(x,t)$, the amount of average flux data required for this inverse problem can be greatly reduced. We first introduce the following coupled system
\begin{align}
\label{utm}
\begin{cases}
\displaystyle\frac{\partial u}{\partial t}-d_1\Delta u+\alpha\cdot\nabla u=l(x)u+b_1(x,t)m+q_1(x,t)
\hspace*{0.5cm} \ &\mathrm{in}\ D_T\medskip\\
\displaystyle \frac{\partial m}{\partial t}-d_2\Delta m+\beta\cdot\nabla m=v(x)m+b_2(x,t)u+q_2(x,t)
\hspace*{0.2cm} &\mathrm{in}\ D_T,\medskip\\
\displaystyle  u=m=0 \hspace*{7.8cm} &\mathrm{on}\ S_T,\medskip\\
\displaystyle  u(x,0)=0, m(x,0)=0\ \hspace*{5.6cm} &\mathrm{in}\ \Omega,\medskip\\
\displaystyle  u>0,\  m>0\ \hspace*{7.2cm} &\mathrm{in}\ D_T.
\end{cases}
\end{align}
The next theorem states the corresponding unique result.
\begin{thm}\label{thm2}
Let $\mathcal{F}^+_{l_i,v_i}$ and $\mathcal{G}^+_{l_i,v_i}$ be the measurement maps associated to \eqref{utm} for $i=1,2$, the weight functions $h,g\in C_0^{\gamma,\gamma/2}(\Gamma_T)$ be nonnegative and nonzero, and $V(t)\in C^1(0,T]$ is positive and $V(0)=0$.  Assume that $v,l\in C^{\gamma}(\overline{\Omega})$ are negative functions, and $V_1=V(t), V_2=V'(t)$. If
\begin{equation}
\label{fcd}
\begin{aligned}
{\mathcal{F}}_{l_1,v_1}^+(b_1^s,b_2^s, q_1^{j,n},q_2^{j,n})&={\mathcal{F}}_{l_2,v_2}^+(b_1^s,b_2^s, q_1^{j,n},q_2^{j,n}),\\
{\mathcal{G}}_{l_1,v_1}^+(b_1^s,b_2^s, q_1^{j,n},q_2^{j,n})&={\mathcal{G}}_{l_2,v_2}^+(b_1^s,b_2^s, q_1^{j,n},q_2^{j,n}),
\end{aligned}
\end{equation}
where $b_i^s\in\mathcal{A}$, $q_i^{j,n}=\phi^n(x) V_j(t)$ with $\phi^n(x)\in \mathcal{B}$, $i,j=1,2$, $s,n\in \mathds{N}$, then we have
$$l_1(x)=l_2(x), v_1(x)=v_2(x)\ in \ \Omega.$$

\end{thm}
\begin{proof} We show the uniqueness of $l(x)$ and the uniqueness of $v(x)$ can be derived similarly. Let $(u_{j,s,n}(x,t;l,v), m_{j,s,n}(x,t;l,v))$ be the solution of \eqref{utm} corresponding to the input functions $b_i^s,q_i^{j,n}$. We set $h_0=h$ on $\Gamma_T$ and $h_0=0$ on $S_T\backslash\Gamma_T$. The $ w(x,t;l)$ is the solution of the following adjoint problem
\begin{eqnarray}
\label{mix-eq-weak1}
\begin{cases}
\begin{aligned}
\displaystyle-\frac{\partial w}{\partial t}-\Delta w-l(x)w-\alpha\cdot \nabla w=0
\hspace*{0.8cm}\ &\text{in}\ \ D_T,\medskip\\
\displaystyle w=h_0\ \hspace*{5.3cm} &\text{on}\ \ S_T,\medskip\\
\displaystyle w(x,T)=0\ \hspace*{4.5cm} &\text{in}\ \ \Omega.\\
\end{aligned}
\end{cases}
\end{eqnarray}
when $j=1$, both sides of the first equation for coupled system \eqref{utm} are multiplied by $w(x,t;l_1)$ and integrated over $D_T$. According to $(\ref{mix-eq-weak1})$ and  Green's formula, we obtain that
\begin{align*}
\int_{0}^T \int_{\Omega}(b_1^sm_{1,s,n}(x,t;l_1)+\phi^{n}(x)V(t))w(x,t;l_1)dxdt
=-\int_0^T\int_{\Lambda}\frac{\partial u_{1,s,n}(x,t;l_1)}{\partial \nu}hdxdt.
\end{align*}
 For $j=2$, we derive
\begin{eqnarray*}
\begin{aligned}
\int_{0}^T\int_{\Omega} (b_1^sm_{2,s,n}(x,t;l_1)+\phi^{n}(x)V'(t))w(x,t;l_1)dxdt=-\int_0^T\int_{\Lambda}\frac{\partial u_{2,s,n}(x,t;l_1)}{\partial \nu}hdxdt.
\end{aligned}
\end{eqnarray*}
Similarly,  multiplying both sides of the first equation for coupled system \eqref{utm} by $w(x,t; l_2)$ for $j=1,2$,  we can see that
\begin{eqnarray*}
\begin{aligned}
\int_{0}^T\int_{\Omega} (b_1^sm_{1,s,n}(x,t;l_2)+\phi^{n}(x)V(t))w(x,t;l_2)dxdt&=-\int_0^T\int_{\Lambda}\frac{\partial u_{1,s,n}(x,t;l_2)}{\partial \nu}hdxdt,\\
\int_{0}^T \int_{\Omega} (b_1^sm_{2,s,n}(x,t;l_2)+\phi^{n}(x)V'(t))w(x,t;l_2)dxdt&=-\int_0^T\int_{\Lambda}\frac{\partial u_{2,s,n}(x,t;l_2)}{\partial \nu}hdxdt.
\end{aligned}
\end{eqnarray*}
From the condition given by \eqref{fcd}, we have
\begin{eqnarray}
\label{vnmu}
\begin{aligned}
&\int_{0}^T\int_{\Omega} (b_1^sm_{j,s,n}(x,t;l_1)+\phi^{n}(x)V_j(t))w(x,t;l_1)dxdt\\
=&\int_{0}^T \int_{\Omega} (b_1^sm_{j,s,n}(x,t;l_2)+\phi^n(x)V_j(t))w(x,t;l_2)dxdt.\\
\end{aligned}
\end{eqnarray}
Notice that $\{b_1^s\}_{s=1}^{\infty}\in\mathcal{A}$ and $\{\phi^n\}_{n=1}^{\infty}\in \mathcal{B}$ are uniformly bounded sequences, as $s\rightarrow\infty$ in \eqref{vnmu}, we find
\begin{align*}
\int_{0}^T\int_{\Omega} \phi^{n}(x)V(t)w(x,t;l_1)dxdt=&\int_{0}^T \int_{\Omega} \phi^{n}(x)V(t)w(x,t;l_2)dxdt,\\
\int_{0}^T\int_{\Omega} \phi^{n}(x)V'(t)w(x,t;l_1)dxdt=&\int_{0}^T\int_{\Omega} \phi^{n}(x)V'(t)w(x,t;l_2)dxdt.
\end{align*}
By the completeness of $\{\phi^n(x)\}_{n=1}^{\infty}$, we obtain that
\begin{eqnarray}
\label{eq3}
\begin{aligned}
\int_{0}^T V(t)w(x,t;l_1)dt=&\int_{0}^T V(t)w(x,t;l_2)dt,\\
\int_{0}^T V'(t)w(x,t;l_1)dt=&\int_{0}^T V'(t)w(x,t;l_2)dt.\\
\end{aligned}
\end{eqnarray}
Multiplying equation (\ref{mix-eq-weak1}) by $V(t)$ and integrating by parts over $(0,T)$, we have
\begin{eqnarray}\small
\label{eq2}
\begin{aligned}
\int_{0}^T-w_t(x,t;l_1)Vdt-\int_0^T\Delta w(x,t;l_1) Vdt-\int_0^T l_1(x)w(x,t;l_1)Vdt-\int_0^T\alpha\cdot \nabla w(x,t;l_1)Vdt=0,\\
\int_{0}^T- w_t(x,t;l_2)Vdt-\int_0^T \Delta w(x,t;l_2)Vdt-\int_0^T l_2(x)w(x,t;l_2)Vdt-\int_0^T\alpha\cdot \nabla w(x,t;l_2)Vdt=0.
\end{aligned}
\end{eqnarray}
Subtraction of two equations in \eqref{eq2} and combine with \eqref{eq3}, it follows that
\begin{eqnarray*}
(l_1-l_2)\int_0^T w(x,t;l_1)V(t)dt=0.
\end{eqnarray*}
By the maximum principle (see \cite{E2009}) for equation \eqref{mix-eq-weak1}, we deduce that $w(x,t;l_1)>0$, then $l_1=l_2$ in $\Omega$. The proof is complete.
\end{proof}

\begin{rem}\label{rem4.2}
For the nonlinear parabolic coupled system \eqref{utmt}, if the unknown coefficients depend only on the spatial variable $x$, a unique identification result can be obtained under the same conditions as the Theorem \ref{thm2}. It's important to note  that in contrast to identifying $l(x,t)$ and $v(x,t)$, only the space segments of injected sources need to form a complete set of $L^2(\overline{\Omega})$, and the time segments only need two modulations. Thus, the measurement data can be substantially reduced.
\end{rem}

\subsection{Uniqueness results for inverse problems associated with the elliptic coupled system}\label{sect:5}
In this subsection, we present the unique identifiable theorem for the coefficients $l(x), v(x)$ of the coupled elliptic system \eqref{umell}.

\begin{thm}
Let $\mathcal{F}^+_{l_i,v_i}$ and $\mathcal{G}^+_{l_i,v_i}$ be the measurement maps associated to \eqref{umell} for $i=1,2$, and the weight functions $h,g\in C_0^{\gamma}(\Lambda)$ be nonnegative and nonzero. Assume that $v,l\in C^{\gamma}(\overline{\Omega})$ are negative functions. If
\begin{equation}\label{g-ste1}
\begin{aligned}
\mathcal{F}_{l_1,v_1}^+(b_1^s,b_2^s,q_1^n,q_2^n)&=\mathcal{F}_{l_2,v_2}^+(b_1^s,b_2^s,q_1^n,q_2^n),\\
\mathcal{G}_{l_1,v_1}^+(b_1^s,b_2^s,q_1^n,q_2^n)&=\mathcal{G}_{l_2,v_2}^+(b_1^s,b_2^s,q_1^n,q_2^n)
\end{aligned}
\end{equation}
for $b_i^s(x)\in \mathcal{A}$, $q_i^n(x)\in \mathcal{B}$ with $s, n\in\mathds{N}$, then it holds that
$$l_1(x)=l_2(x), v_1(x)=v_2(x)\  in\  \Omega.$$
\end{thm}

\begin{proof}
 We only show the uniqueness of $l(x)$, the uniqueness of $v(x)$ can be derived similarly. Let $(u_{s,n}(x;l,v), m_{s,n}(x;l,v))$ be the solution of \eqref{umell} corresponding to the input functions $b_i^s,q_i^{n}$. Set $h_0=h$ on $\Lambda$ and $h_0=0$ on $\partial\Omega\backslash\Lambda$, and $w(x;l)$ is the solution to the following adjoint problem
\begin{eqnarray}
\label{W}
\begin{cases}
\displaystyle -d_1\Delta w-l(x)w=0\hspace*{2.0cm} \ &\text{in}\ \ \Omega,\medskip\\
\displaystyle w=h_0\hspace*{3.2cm} \ &\text{on}\ \ \partial\Omega.
\end{cases}
\end{eqnarray}
Both sides of the first equation of the system \eqref{um} are multiplied by $w(x; l_1), w(x; l_2)$, respectively, by using the Green's formula and combine with \eqref{g-ste1}, we compute
\begin{eqnarray}
\label{umell12}
\begin{aligned}
&\int_{\Omega}\bigg{(}-b_1^s\sum_{k=1}^Mu_{s,n}^km_{s,n}^{M-k+1}+q_1^n(x)\bigg{)}w(x;l_1)dx\\
=&\int_{\Omega}\bigg{(}-b_1^s\sum_{k=1}^Mu_{s,n}^km_{s,n}^{M-k+1}+q_1^n(x)\bigg{)}w(x;l_2)dx.
\end{aligned}
\end{eqnarray}
 Due to $\{b_1^s\}_{s=1}^{\infty}\in \mathcal{A}$ and $\{q_1^n\}_{n=1}^{\infty}\in\mathcal{B}$ are uniformly bounded sequences, let $s\rightarrow\infty$ in \eqref{umell12}, it follows from the completeness of $\{q_1^n(x,t)\}_{j=1}^{\infty}$ that
$$w(x;l_1)=w(x;l_2) \ \mathrm{for}\  x\in\Omega.$$
The equations in \eqref{W} corresponding to $l_1, l_2$, respectively, followed by subtraction, we find
\begin{eqnarray*}
(l_1-l_2)w(x;l_1)=0\ \mathrm{for}\ x\in \Omega.
\end{eqnarray*}
 Since $w(x;l_1)>0$ for $x\in \Omega$, then $l_1=l_2$ in $\Omega$. The proof is complete.
\end{proof}

\section{Inverse problems associated with the coupled parabolic system with Neumann boundary condition}\label{sect:5}
In this section, we study an inverse problem associated with the coupled system \eqref{utmt} with a homogeneous Neumann boundary, i.e.,
 \[\frac{\partial u}{\partial\nu}=\frac{\partial m}{\partial \nu}=0\ \ \mathrm{on}\  S_T.\]
The homogeneous Neumann boundary condition usually indicates that this system is self-contained with zero population flux across the boundary. Under this boundary condition, we utilize the average flux data on a partial boundary to uniquely identify/determine the coefficients $l(x,t)$ and $v(x,t)$ in the coupled system as follows
 \begin{align}
 \small
\label{utmtN}
\begin{cases}
\displaystyle\frac{\partial u}{\partial t}-d_1\Delta u+\alpha\cdot\nabla u=l(x,t)u+a_1\bigg{(}-u^2-\sum_{k=1}^Mu^km^{M-k+1}\bigg{)}+q_1
\ \ \hspace*{1.0cm} &\mathrm{in}\  D_T,\medskip\\
\displaystyle \frac{\partial m}{\partial t}-d_2\Delta m+\beta\cdot\nabla m=v(x,t)m+a_2\bigg{(}-m^2-\sum_{k=1}^Mu^km^{M-k+1}\bigg{)}+q_2
\hspace*{0.65cm} &\mathrm{in}\ D_T,\medskip\\
\displaystyle  \frac{\partial u}{\partial\nu}=\frac{\partial m}{\partial\nu}=0\hspace*{10.1cm} &\mathrm{on}\  S_T,\medskip\\
u(x,0)=0,\  m(x,0)=0\ \hspace*{8.35cm} &\mathrm{in}\ \Omega,\medskip\\
\displaystyle u>0,\  m>0\ \hspace*{8.4cm} &\mathrm{in}\ D_T.
\end{cases}
\end{align}
We would like to emphasize that the measurement maps associated with a coupled system \eqref{utmtN} are different from the measurement maps for the coupled system with the Dirichlet boundary,  which we denote as  follows:
\begin{equation*}
\begin{aligned}
\mathcal{F}^{\mathcal{N+}}_{l,v}(a_1,a_2,q_1,q_2):=\int_{\Gamma_T}u(x,t)h(x,t)dxdt,\\
\mathcal{G}^{\mathcal{N+}}_{l,v}(a_1,a_2,q_1,q_2):=\int_{\Gamma_T} m(x,t)g(x,t)dxdt,
\end{aligned}
\end{equation*}
where $h,g$ are still the weight functions. Next, we show the uniqueness theorem associated with the system \eqref{utmtN}.

\begin{thm}Let $\mathcal{F}^{\mathcal{N+}}_{l_i,v_i}$ and $\mathcal{G}^{\mathcal{N+}}_{l_i,v_i}$ be the measurement maps associated to \eqref{utmtN} for $i=1,2$, and the weight functions $h,g\in C_0^{\gamma,\gamma/2}(\Gamma_T)$ be nonnegative and nonzero. Assume that $v, l\in C^{\gamma,\gamma/2}(\overline{D_T})$ are negative functions. If
\begin{equation}\label{umfg1}
\begin{aligned}
\mathcal{F}^{\mathcal{N+}}_{l_1,v_1}(a_1^s,a_2^s,q_1^n,q_2^n)&=\mathcal{F}^{\mathcal{N+}}
_{l_2,v_2}(a_1^s,a_2^s,q_1^n,q_2^n),\\
\mathcal{G}^{\mathcal{N+}}_{l_1,v_1}( a_1^s,a_2^s,q_1^n,q_2^n)&=\mathcal{G}^{\mathcal{N+}}_{l_2,v_2}(a_1^s,a_2^s,q_1^n,q_2^n)
\end{aligned}
\end{equation}
for $a_i^s\in\mathcal{A}$, and $q_i^n\in \mathcal{B}$ with  $s, n\in\mathds{N}$, then it holds that
 $$l_1(x,t)=l_2(x,t), v_1(x,t)=v_2(x,t)\  in\  D_T.$$
\end{thm}

\begin{proof}
Since the process of obtaining the uniqueness of $v(x,t)$ is similar to the one of $l(x,t)$, we only show the uniqueness of $l(x,t)$. Let $(u_{s,n}(x;l,v), m_{s,n}(x;l,v))$ be the solution of \eqref{utmtN} corresponding to the input functions $a_i^s,q_i^{n}$. We set $h_0=h$ on $\Gamma_T$ and $h_0=0$ on $S_T\backslash\Gamma_T$, and the function $ w(x,t;l_1)$ as the solution of the following adjoint problem
\begin{eqnarray}
\label{w_n}
\begin{aligned}
\begin{cases}
-\displaystyle\frac{\partial w}{\partial t}-d_1\Delta w-\alpha\cdot \nabla w-l_1(x,t)w=0\ \hspace*{0.5cm} &\text{in}\ \ D_T,\medskip\\
\displaystyle \frac{\partial w}{\partial \nu}=h_0\ \hspace*{5.7cm} &\text{on}\ \ S_T,\medskip\\
\displaystyle w(x,T)=0\ \hspace*{5.3cm} &\text{in}\ \ \Omega.\\
\end{cases}
\end{aligned}
\end{eqnarray}
Multiplying both sides of the first equation in the system $(\ref{utmtN})$ by $w(x,t;l_1), w(x,t;l_2)$ respectively, integrating over $D_T$ and using Green's formula, we compute
\begin{eqnarray*}\small
\int_{0}^T \int_{\Omega}\bigg{(}-a_1^su_{s,n}^2-a_1^s\sum_{k=1}^Mu_{s,n}^km_{s,n}^{M-k+1}+q_1^{n}\bigg{)}w(x,t;l_1)dxdt
=\int_0^T\int_{\Lambda} u_{s,n}(x,t;l_1,v_1)h(x,t)dxdt,\\
\int_{0}^T \int_{\Omega}\bigg{(}-a_1^su_{s,n}^2-a_1^s\sum_{k=1}^Mu_{s,n}^km_{s,n}^{M-k+1}+q_1^{n}\bigg{)}w(x,t;l_2)dxdt
=\int_0^T\int_{\Lambda} u_{s,n}(x,t;l_2,v_2)h(x,t)dxdt.
\end{eqnarray*}
From \eqref{umfg1}, we have
\begin{equation}\label{uNmN}
\begin{aligned}
&\int_{0}^T \int_{\Omega}\bigg{(}-a_1^s u_{s,n}^2-a_1^s \sum_{k=1}^Mu_{s,n}^km_{s,n}^{M-k+1}+q_1^{n}\bigg{)}w(x,t;l_1)dxdt\\
&=\int_{0}^T \int_{\Omega}\bigg{(}-a_1^su_{s,n}^2-a_1^s \sum_{k=1}^Mu_{s,n}^km_{s,n}^{M-k+1}+q_1^{n}\bigg{)}w(x,t;l_2)dxdt.
\end{aligned}
\end{equation}
Notice that $\{a_1^s\}_{s=1}^\infty\in\mathcal{A}, \{q_1^n\}_{n=1}^\infty\in\mathcal{B}$ are uniformly bounded and the solution $(u_{s,n}, m_{s,n})$ of the system \eqref{utmtN} have positive constant upper solutions. Then, taking $s\rightarrow \infty$ in \eqref{uNmN}, it follows from the completeness of $\{q_1^n(x,t)\}_{j=1}^{\infty}$ that
\begin{align}\label{w22}
w(x,t;l_1)=w(x,t;l_2)\ \ \mathrm{for} \ (x,t)\in D_T.
\end{align}
The equations for \eqref{w_n} correspond to $l_1(x,t)$ and $l_2(x,t)$, respectively, subtracted from each other, and combined with \eqref{w22}, we have
\begin{eqnarray*}
(l_1-l_2) w(x,t;l_1)=0\ \ \mathrm{for} \ (x,t)\in D_T.
\end{eqnarray*}
 By the maximum principle, we obtain $w(x,t;l_1)> 0$ in $D_T$, thus $l_1=l_2$ in $D_T$. The proof is complete.
\end{proof}

\begin{rem}
In the coupled system \eqref{utmtN}, if the coefficients $l,v$ depend only on the spatial variable $x$, the uniqueness result for determining the $l(x)$ and $v(x)$ can be obtained under the same conditions of Theorem \ref{thm2}. It is worth noting that the source terms $q_1,q_2$ still require the forms of separated variables, and the proof is similar to that of the coupled system with the homogeneous Dirichlet boundary in Theorem \ref{thm2}.
\end{rem}

\section*{Acknowledgments}
 The work of Hongyu Liu is supported by the Hong Kong RGC General Research Funds (projects 12302919, 12301420, and 11300821), the NSFC/RGC Joint Research Fund (project $\rm N_{-}CityU101/21$), the France-Hong Kong ANR/RGC Joint Research Grant, AHKBU203/19. The work of Guang-Hui Zheng is supported by the NSF of China (12271151) and the NSF of Hunan (2020JJ4166).


\begin{thebibliography}{99}
\addcontentsline{toc}{section}{References}


\bibitem{AAR2004}
{\sc E. Allman, E. Allman  and J. Rhodes}. {\em Mathematical models in biology: an introduction}. Cambridge University Press, 2004.

\bibitem{ACR2007}
{\sc A. Ambrosetti,  E. Colorado and D. Ruiz}. {\em Multi-bump solitons to linearly coupled systems of nonlinear Schr$\ddot{o}$dinger equations}. Calc. Var. Par. Dif., 30(1): 85-112, 2007.

\bibitem{AL2023}
{\sc A. Aparicio and Y. Liu}. {\em Distinct dynamic phases observed in bacterial microcosms}. Engineering Microbiology.  3(1): 100063, 2023.

\bibitem{BB1984}
{\sc J. Blat and K. Brown}. {\em Bifurcation of steady-state solutions in predator-prey and competition systems}. Proc. Roy. Soc. Edinburgh Sect. A, 97: 21-34, 1984.

\bibitem{CGN2013}
{\sc  M. Cristofol, P. Gaitan, K. Niinim$\ddot{\mathrm{a}}$ki, et al}. {\em Inverse problem for a coupled parabolic system with discontinuous conductivities: one-dimensional case}. Inverse Probl. Imag., 2013.


\bibitem{DLS2002}
{\sc M. Delgado, J. L$\mathrm{\acute{o}}$pez-G$\mathrm{\acute{o}}$mez and A. Su$\mathrm{\acute{a}}$rez}. {\em On the symbiotic Lotka-Volterra model with diffusion and transport effects}. J. Differ. Equations, 160(1): 175-262, 2000.

\bibitem{DZ2021}
{\sc M. Ding and G. Zheng}, {\em Determination of the reaction coefficient in a time dependent nonlocal diffusion process}. Inverse Probl., 37(2): 025005, 2021.


\bibitem{DEM2009}
{\sc P. Dostert, Y. Efendiev and B. Mohanty}, {\em Efficient uncertainty quantification techniques
in inverse problems for Richards equation using coarse-scale simulation models}, Adv. Water Resour., 32(3): 329-339, 2009.

\bibitem{DY2019}
{\sc  F. Dou and M. Yamamoto}, {\em Logarithmic stability for a coefficient inverse problem of coupled Schr\"{o}dinger equations}. Inverse Problems, 35(7): 075006, 2019.

\bibitem{DY2002}
{\sc Y. Du}, {\em Effects of a degeneracy in the competition model: part I. Classical and generalized steady-state solutions}. J. Differ. Equations, 181(1): 92-132, 2002.


\bibitem{E2009}
{\sc L. Evans}, {\em Partial differential equations}. Instructor 67, 2009.

\bibitem{F1977}
{\sc A. Fbwler}, {\em Convective diffusion on an enzyme reaction}. SIAM J. Appl. Math., 33: 289-297, 1977.

\bibitem{GSS2006}
{\sc V. Galaktionov and S. Svirshchevskii}, {\em Exact solutions and invariant subspaces of nonlinear partial differential equations in mechanics and physics}. Chapman and Hall/CRC, 2006.

\bibitem{GNH2009}
{\sc D. Gaston, C. Newman, G. Hansen, et al}. {\em MOOSE: A parallel computational framework for coupled systems of nonlinear equations}. Nuclear Engineering and Design,  239(10): 1768-1778, 2009.

\bibitem{HAM2022}
{\sc J. Hu, D. Amor, M. Barbier, et al}. {\em Emergent phases of ecological diversity and dynamics mapped in microcosms}. Science, 378(6615): 85-89, 2022.


\bibitem{LLX2009}
{\sc Y. Lin, H. Liu, X. Liu and S. Zhang}. {\em Simultaneous recoveries for semilinear parabolic systems}. Inverse Problems, 38(11): 115006, 2022.

\bibitem{LL1988}
{\sc L. Li}. {\em Coexistence theorems of steady-states for predator-prey interacting systems}. Trans. Amer. Math. Soc., 305: 143-166, 1988.

\bibitem{LL2023}
{\sc H. Liu and C. Lo}, {\em Determining a parabolic system by boundary observation of its non-negative solutions with applications}, arXiv:2303.13045, 2023.

\bibitem{LMZ2022}
{\sc H. Liu, C. Mou and S. Zhang}. {\em Inverse problems for mean field games}, arXiv:2205.11350, 2022.

\bibitem{LZ2022}
{\sc H. Liu and S. Zhang}, {\em On an inverse boundary problem for mean field games}, arXiv:2212.09110, 2022.

\bibitem{MSR2010}
{\sc  P. Manimozhi,  A. Subbiah and L. Rajendran}. {\em Solution of steady-state substrate concentration in the action of biosensor response at mixed enzyme kinetics}. Sensors and Actuators B: Chemical,  147(1): 290-297,
2010.
\bibitem{M1975}
{\sc R. May}. {\em Biological populations obeying difference equations: stable points, stable cycles, and chaos}. J. Theor. Biol.,  51(2): 511-524,  1975.

\bibitem{May1976}
{\sc R. May}. {\em Simple mathematical models with very complicated dynamics}. Nature,  261, 459-467, 1976.

\bibitem{MRMG2009}
{\sc N. Mcenroe, N. Roulet, T. Moore and M. Garneau}. {\em Do pool surface areaand depth
control $CO_2$ and $CH_4$ fluxes from an ombrotrophic raised bog, James Bay, Canada\ ?}. J.
Geophys. Res. Biogeosci. 114($\mathrm{G_1}$), 2009.


\bibitem{M1977}
{\sc J. Murry}. {\em Lectures on Nonlinear Differential equations: Models
in Biology}, Clarendon Press, Oxford, 1977.


\bibitem{P2012}
{\sc C. Pao}. {\em Nonlinear parabolic and elliptic equations}. Springer Science Business Media, 2012.


\bibitem{RC2012}
{\sc  L. Roques and M. Cristofol}. {\em The inverse problem of determining several coefficients in a nonlinear Lotka-Volterra system}. Inverse Probl., 28(7): 075007, 2012.

\bibitem{R2013}
{\sc T. Roub\'{\i}\v{c}ek}. {\em Nonlinear partial differential equations with applications}. Springer Science  Business Media, 2013.


\bibitem{Se2009}
{\sc  R. Seydel}. {\em Practical bifurcation and stability analysis}. Springer Science  Business Media, 2009.

\bibitem{S2018}
{\sc S. Strogatz}. {\em Nonlinear dynamics and chaos: with applications to physics, biology, chemistry, and engineering}. CRC press, 2018.


\bibitem{ZD2020}
{\sc G. Zheng and M. Ding}, {\em Identification of the degradation coefficient for an anomalous diffusion process in hydrology}. Inverse Probl., 36(3): 035006, 2020.


\end{thebibliography}
\end{document}